\documentclass[11pt]{amsart}
\usepackage{graphicx}
\usepackage[T1]{fontenc}
\usepackage{setspace}
\usepackage[hidelinks]{hyperref}
\usepackage{mathrsfs}
\usepackage{geometry}
\newtheorem{thm}{Theorem}[section]
\newtheorem{cor}[thm]{Corollary}

\theoremstyle{definition}
\newtheorem{defn}[thm]{Definition}
\theoremstyle{remark}

\theoremstyle{definition}

\begin{document}
\title[Lamperti's recurrence of stochastic sequences]{Some notes on Lamperti's recurrence of stochastic sequences}%
\author{Vyacheslav M. Abramov}%
\address{24 Sagan Drive, Cranbourne North, Victoria 3977, Australia}%

\email{vabramov126@gmail.com}%

\thanks{ORCID: 0000-0002-9859-100X}%

\begin{abstract}
The present study provides another look on Lamperti's theorem on recurrence or transience of stochastic sequences. We establish a connection between Lamperti's theorem and the recent result by the author in [V. M. Abramov, \emph{A new criterion for recurrence of Markov chains with infinitely countable set of states,} Theor. Probab. Math. Stat. \textbf{112} (2025), 1--15].
\end{abstract}

\maketitle

\section{Introduction}

The study of recurrence or transience of Markov processes has a long history going back to the fifties of the last century. Particular Markov chains were studied by Gillis \cite{G}, Harris \cite{H}, Hodges and Rosenblatt \cite{HR}, Karlin and McGregor \cite{KM, KM1}. The well-known book by Chung \cite{K} published in 1960 also discussed Markov chains with stationary transition probabilities, where the particular chain of the birth-and-death type was studied in \cite[Part 1, \S{12}]{K}.

In 1960 Lamperti \cite{L} proved one of the most important result that establishes the conditions for recurrence or transience for the class of discrete time stochastic processes that generalizes Markov chains. His conditions of recurrence or transience required the integrability assumption for stochastic processes (see \eqref{2} below), and was based on the asymptotic relationship between the first and second conditional moments of the process.

Recently Abramov \cite{A} obtained alternative conditions for recurrence or transience of Markov chains with infinitely countable set of states. The necessary and sufficient condition obtained there was closely related to that one given for birth-and-death type Markov chains in \cite[Part 1, \S{12}]{K}. There was no requirement on the existence of the second moments, but it was required that a Markov chain formed a connected domain. The definition of connected domain is given in \cite{A}.

The aim of this study is another look at the results of Lamperti in order to build the bridge between these two studies.

In the existing literature, we did not find alternative studies on recurrence or transience of Markov chains. For more clarity, we refer the known relatively recent books on the theory of Markov chains \cite{N} and the stability or ergodicity of continuous time Markov processes \cite{MT}.

Let $\{X_n,\mathscr{F}_n\}$, $\mathscr{F}_n=\sigma(X_n, X_{n-1},\ldots,X_0)$, be a nonnegative stochastic sequence satisfying the property
\begin{equation}\label{1}
\mathsf{P}\{\limsup_{n\to\infty}X_n=\infty\}=1,
\end{equation}
(Further, the notation for stochastic sequence is provided without indication of the filtration.)
The stochastic sequence $X_n$ is said to be recurrent, if there exists $r<\infty$ such that
$
\mathsf{P}\{\liminf_{n\to\infty}X_n\leq r\}=1,
$
and transient if
$
\mathsf{P}\{\lim_{n\to\infty}X_n=\infty\}=1,
$
e. g., \cite[pp. 316--317]{L}.

Note that the definition of recurrence given here for stochastic sequences is not traditional, since it does not imply that all states are visited
infinitely often. But it is Harris recurrence, where the subset $[0, r]$ is assumed to be visited infinitely often \cite{MT}. The study that conducted in \cite{A} used the definition of recurrence in the traditional sense. This however makes no difference for the result of our study; the different definitions for recurrence or transience finally lead to equivalent conditions for recurrence or transience, that in particular is confirmed by the present study as well.

Assume that
\begin{equation}\label{2}
\mathsf{E}\{|X_{n+1}-X_n|^{2+\epsilon}~|~\mathscr{F}_n\}\leq B<\infty \ \text{for some positive}\ \epsilon,
\end{equation}
and denote
\begin{eqnarray*}
\overline{\mu}(x)&=&\mathrm{ess} \sup\mathsf{E}\{X_{n+1}-X_n~|~X_n=x, X_{n-1}, X_{n-2},\ldots, X_0\},\\
\underline{\mu}(x)&=&\mathrm{ess} \inf\mathsf{E}\{X_{n+1}-X_n~|~X_n=x, X_{n-1}, X_{n-2},\ldots, X_0\},\\
\overline{v}(x)&=&\mathrm{ess} \sup\mathsf{E}\{(X_{n+1}-X_n)^2~|~X_n=x, X_{n-1}, X_{n-2},\ldots, X_0\},\\
\underline{v}(x)&=&\mathrm{ess} \inf\mathsf{E}\{(X_{n+1}-X_n)^2~|~X_n=x, X_{n-1}, X_{n-2},\ldots, X_0\}.
\end{eqnarray*}
(To avoid unnecessary complications, here and later we indicate the only dependence of $x$ and do not indicate the dependence of the history of the process. This is justified by the nature of asymptotic results, where the convergence in $x$ is assumed to be uniform and independent of the history.)
\smallskip

The following theorem is due to Lamperti \cite{L}.

\begin{thm}\label{thm_Lamperti} (Lamperti \cite{L}.)
Let the nonnegative stochastic sequence $X_n$ satisfy \eqref{1} and \eqref{2}, and, as $x\to\infty$,
\[
\overline{\mu}(x)\leq\frac{\underline{v}(x)}{2x}+O\left(\frac{1}{x^{1+\epsilon}}\right).
\]
Then the stochastic sequence is recurrent. If instead for some $\theta>1$ and almost all large $x$,
\[
\underline{\mu}(x)\geq\frac{\theta\overline{v}(x)}{2x},
\]
then the stochastic sequence is transient.
\end{thm}

Unfortunately, this quite general Theorem \ref{thm_Lamperti} addresses the question about recurrence or transience for a small class of stochastic sequences, in which $\underline{\mu}(x)\neq\overline{\mu}(x)$ or/and $\underline{v}(x)\neq\overline{v}(x)$ for all large $x$. We tested the two examples considered in \cite[Sect. 4]{A}, and for both of them, despite of their clarity, Theorem \ref{thm_Lamperti} did not give any answer.
 Therefore, we will only discuss its corollary given below,  the proof of which given in \cite{L} was based on Doob's convergence theorem for semimartingales \cite[Chapt. 7]{D}.

\begin{cor}\label{cor_Lamperti} (Lamperti \cite{L}.)
Assume that a nonegative stochastic sequence $X_n$ satisfies \eqref{1}, and assume that for $x\to\infty$, we almost surely have
\begin{eqnarray}
\mu(x)&=&\mathsf{E}\{X_{n+1}-X_n~|~X_n=x, X_{n-1},\ldots,X_0\}=\frac{\xi}{2x}+O\left(\frac{1}{x^{1+\epsilon}}\right), \label{4.2}\\
v(x)&=&\mathsf{E}\{(X_{n+1}-X_n)^2~|~X_n=x,  X_{n-1},\ldots,X_0\}=r^2+O\left(\frac{1}{x^{\epsilon}}\right),\nonumber\\
&&\text{where}\ \xi \ \text{is a positive value}.\nonumber
\end{eqnarray}
Then the stochastic sequence is recurrent, if for almost all $x\to\infty$,
\[
\mu(x)\leq \frac{v(x)}{2x}+O\left(\frac{1}{x^{1+\epsilon}}\right).
\]
If instead for some $\theta>1$ and  almost all large $x$,
\[
\mu(x)\geq\frac{\theta v(x)}{2x},
\]
then the stochastic sequence is transient.
\end{cor}

The rest of this study is organized as follows. In Section \ref{S2}, we prove a new theorem that is similar to Corollary \ref{cor_Lamperti}.
In Section  \ref{S3}, we discuss our findings in Section \ref{S2} and show their connection with the main result in \cite{A}. The appendix contains a supplementary material related to the given study.

 \section{Main result and its proof}\label{S2}

The conditions of recurrence and transience for stochastic sequences given by Corollary \ref{cor_Lamperti}  are expressed explicitly via the limits with probability 1 of $x\mu(x)$ and $v(x)$  as $x\to\infty$, and, hence, the statement of the corollary is defined for the classes of stochastic sequences where these limits exist and takes the given values.

The new result proved in this section suggests the representation that looks slightly more complicated than that in Corollary \ref{cor_Lamperti}. As well, our conditions are more specified compared to the assumptions of Corollary \ref{cor_Lamperti}. On the other hand, for the class of processes considered here we avoid the use of the almost sure limit of $v(x)$ as $x\to\infty$.%

The advantage of our result is that it helps to understand clearer the connection with the recent result obtained in \cite{A} as well as a new phenomena in the behaviour of the considered stochastic sequences or Markov chains.

\subsection*{Assumptions}
We assume that for almost all large $x$
\begin{eqnarray}
\mathsf{E}\{X_{n+1}-X_n~\big|~X_n=x, X_{n+1}<X_n, X_{n-1},\ldots, X_0\}&=&-A^*+\frac{\xi^*}{x}+O\left(\frac{1}{x^{1+\epsilon}}\right),\label{4.3}\\
\mathsf{E}\{X_{n+1}-X_n~\big|~X_n=x, X_{n+1}>X_n, X_{n-1},\ldots, X_0\}&=&A^{**}+\frac{\xi^{**}}{x}+O\left(\frac{1}{x^{1+\epsilon}}\right),\label{4.4}
\end{eqnarray}
and
\begin{eqnarray}
\mathsf{P}\{X_{n+1}<X_n~|~X_n=x, X_{n-1},\ldots, X_0\}&=&q+O\left(\frac{1}{x^{1+\epsilon}}\right),\label{4.5}\\
\mathsf{P}\{X_{n+1}>X_n~|~X_n=x, X_{n-1},\ldots, X_0\}&=&p+O\left(\frac{1}{x^{1+\epsilon}}\right), \quad q+p\leq1\label{4.6}
\end{eqnarray}
for positive $A^*, A^{**}$ satisfying $qA^*=pA^{**}$ and arbitrary parameters $\xi^*, \xi^{**}$ satisfying $q\xi^*+p\xi^{**}>0$.

Without loss of generality one can assume that $p+q=1$ since the processes with this property have the same classification as those with $0<p+q<1$.
(It is easy to see that it is consistent with Corollary \ref{cor_Lamperti}, since the assumption
\[
\lim_{x\to\infty}\mathsf{P}\{X_{n+1}=X_n~|~X_n=x, X_{n-1},\ldots, X_0\})>0
\]
leads to the same classification of recurrence or transience as that in the case
\[
\lim_{x\to\infty}\mathsf{P}\{X_{n+1}=X_n~|~X_n=x, X_{n-1},\ldots, X_0\})=0
\]
due to the definition of $\mu(x)$ and $v(x)$.)

Our goal is to prove the following theorem.

\begin{thm}\label{thm1}
Assume that \eqref{1}, integrability condition \eqref{2} and assumptions \eqref{4.3}, \eqref{4.4}, \eqref{4.5} and \eqref{4.6} are satisfied.
Then the stochastic sequence is recurrent if, as $x\to\infty$, we almost surely have
\[
\begin{aligned}
&\frac{p\mathsf{E}\{X_{n+1}-X_n~|~X_n=x, X_{n+1}>X_n, X_{n-1},\ldots, X_0\}}{q\mathsf{E}\{X_{n}-X_{n+1}~|~X_n=x, X_{n+1}<X_n, X_{n-1},\ldots, X_0\}}\\
&\leq1+\frac{R}{x}+O\left(\frac{1}{x^{1+\epsilon}}\right),
\end{aligned}
\]
where $R(x)=\mathsf{E}\{|X_{n+1}-X_n|~\big|~X_{n}=x, X_{n-1},\ldots, X_0\}$, and $R$ is the limit with probability $1$ of $R(x)$ as $x\to\infty$.
If instead
\[
\begin{aligned}
&\frac{p\mathsf{E}\{X_{n+1}-X_n~|~X_n=x, X_{n+1}>X_n, X_{n-1},\ldots, X_0\}}{q\mathsf{E}\{X_{n}-X_{n+1}~|~X_n=x, X_{n+1}<X_n, X_{n-1},\ldots, X_0\}}\\
&\geq1+\frac{R\theta}{x},
\end{aligned}
\]
for some $\theta>1$ and almost all large $x$, then the stochastic sequence is transient.
\end{thm}

\begin{proof}

The class of stochastic sequences $X_n$ satisfying the conditions
\begin{eqnarray*}
\mathsf{E}\{X_{n+1}-X_n~|~X_n=x, X_{n-1},\ldots, X_0\}&=&\frac{\xi}{2x}+O\left(\frac{1}{x^{1+\epsilon}}\right),\\
\mathsf{E}\{(X_{n+1}-X_n)^2~|~X_n=x, X_{n-1},\ldots, X_0\}&=&r^2+O\left(\frac{1}{x^{\epsilon}}\right)
\end{eqnarray*}
for almost all large $x$,  will be denoted by $\frak{A}_{\xi,r^2}$. According to Corollary \ref{cor_Lamperti}, the class of stochastic sequences $X_n$ must have the same classification as the class $X_n^\prime=r^{-1}X_n$, for which
\begin{eqnarray*}
\mathsf{E}\{X_{n+1}^\prime-X_n^\prime~|~X_n^\prime=x, X_{n-1}^\prime,\ldots, X_0^\prime\}&=&\frac{\eta}{2x}+O\left(\frac{1}{x^{1+\epsilon}}\right), \quad \eta=\frac{\xi}{r^2},\\
\mathsf{E}\{(X_{n+1}^\prime-X_n^\prime)^2~|~X_n^\prime=x, X_{n-1}^\prime,\ldots, X_0^\prime\}&=&1+O\left(\frac{1}{x^{\epsilon}}\right).
\end{eqnarray*}
Indeed, for the stochastic sequence $X_n^\prime$ for almost all large $x$ we have
\begin{equation*}
\begin{aligned}
&\mathsf{E}\{X_{n+1}^\prime-X_n^\prime~|~X_n^\prime=x, X_{n-1}^\prime,\ldots, X_0^\prime\}\\
&\quad=\frac{1}{r}\mathsf{E}\{X_{n+1}-X_n~|~X_n=rx, X_{n-1},\ldots, X_0\}\\
&\quad=\frac{\xi}{2r^2x}+O\left(\frac{1}{x^{1+\epsilon}}\right)=\frac{\eta}{2x}+O\left(\frac{1}{x^{1+\epsilon}}\right),
\end{aligned}
\end{equation*}

\begin{equation*}
\begin{aligned}
&\mathsf{E}\{(X_{n+1}^\prime-X_n^\prime)^2~|~X_n^\prime=x, X_{n-1}^\prime,\ldots, X_0^\prime\}\\
&\quad=\mathsf{E}\{(r^{-1}X_{n+1}-r^{-1}X_n)^2~|~r^{-1}X_n=x, X_{n-1},\ldots, X_0\}\\
&\quad=\frac{r^2}{r^2}+O\left(\frac{1}{x^{\epsilon}}\right)=1+O\left(\frac{1}{x^{\epsilon}}\right).
\end{aligned}
\end{equation*}
Hence
\[
x~\frac{\mathsf{E}\{X_{n+1}-X_n~|~X_n=x, X_{n-1},\ldots, X_0\}}{\mathsf{E}\{(X_{n+1}-X_n)^2~|~X_n=x, X_{n-1},\ldots, X_0\}}
\]
and
\[
x~\frac{\mathsf{E}\{X_{n+1}^\prime-X_n^\prime~|~X_n^\prime=x, X_{n-1}^\prime,\ldots, X_0^\prime\}}{\mathsf{E}\{(X_{n+1}^\prime-X_n^\prime)^2~|~X_n^\prime=x, X_{n-1}^\prime,\ldots, X_0^\prime\}}
\]
have the same limit as $x\to\infty$.

Therefore, it is reasonable to first study the stochastic sequences $X_n$ belonging to the class $\frak{A}_{\xi,1}$, that is, the class of normalized stochastic sequences, with the following extension of the results to other required classes of stochastic sequences.

Let $\frak{A}_{\xi,1}$ be the class of normalized stochastic sequences, and let $X_n$ belong to this class. Let us first find the asymptotic behavior of the ratio
\[
\frac{p\mathsf{E}\{X_{n+1}-X_n~|~X_n=x, X_{n+1}>X_n, X_{n-1},\ldots, X_0\}}{q\mathsf{E}\{X_{n}-X_{n+1}~|~X_n=x, X_{n+1}<X_n, X_{n-1},\ldots, X_0\}}
\]

With \eqref{4.3} and \eqref{4.4}, the assumptions that as $x\to\infty$,
\begin{equation}\label{4.8}
\begin{aligned}
\mathsf{E}\{X_{n+1}-X_n~|~X_n=x, X_{n-1},\ldots, X_0\}=\frac{\xi}{2x}+O\left(\frac{1}{x^{1+\epsilon}}\right),
\end{aligned}
\end{equation}
\begin{equation}\label{4.9}
\begin{aligned}
\mathsf{E}\{(X_{n+1}-X_n)^2~\big|~X_{n}=x, X_{n-1},\ldots, X_0\}
=1+O\left(\frac{1}{x^{\epsilon}}\right)
\end{aligned}
\end{equation}
for almost all large $x$, generally implies:
\begin{eqnarray}
\left(\frac{q\xi^*+p\xi^{**}}{x}\right)&=&\frac{\xi}{2x},\nonumber\\
A^*q&=&A^{**}p,\label{4.7}
\end{eqnarray}
and
\[
\begin{aligned}
\mathsf{E}\{|X_{n+1}-X_n|~\big|~X_n=x, X_{n-1},\ldots, X_0\}&=qA^*\left(1-\frac{q\xi^*}{x}\right)+pA^{**}\left(1+\frac{p\xi^{**}}{x}\right)+O\left(\frac{1}{x^{1+\epsilon}}\right)\\
&=pA^{**}\left(2+\frac{p\xi^{**}}{x}-\frac{q\xi^*}{x}\right)+O\left(\frac{1}{x^{1+\epsilon}}\right)\\
&=2pA^{**}+O\left(\frac{1}{x}\right),
\end{aligned}
\]
due to \eqref{4.7}.

Hence, on the one hand we have \eqref{4.8} and \eqref{4.9}, and on the other hand
\[
\begin{aligned}
&\frac{p\mathsf{E}\{X_{n+1}-X_n~|~X_n=x, X_{n+1}> X_n, X_{n-1},\ldots, X_0\}}{q\mathsf{E}\{X_{n}-X_{n+1}~|~X_n=x, X_{n+1}< X_n, X_{n-1},\ldots, X_0\}}\\
&\quad =\frac{pA^{**}+\frac{p\xi^{**}}{x}+O\left(\frac{1}{x^{1+\epsilon}}\right)}{pA^{**}-\frac{q\xi^*}{x}+O\left(\frac{1}{x^{1+\epsilon}}\right)}\\
&\quad = 1+\frac{q\xi^*+p\xi^{**}}{pA^{**}x}+O\left(\frac{1}{x^{1+\epsilon}}\right)\\
&\quad = 1+\frac{\xi}{2pA^{**}x}+O\left(\frac{1}{x^{1+\epsilon}}\right)
\end{aligned}
\]
for almost all large $x$.

It follows from Corollary \ref{cor_Lamperti} that the stochastic sequence $X_n$ is recurrent, if $\xi\leq1+O\big(x^{-1-\epsilon}\big)$. If instead $\xi\geq\theta>1$, then the stochastic sequence $X_n$ is transient. That is, if we define $\eta:=\xi/(2pA^{**})$, then, according to the above, the condition for recurrence of $X_n$ is $\eta\leq1/(2pA^{**})+O\big(x^{-1-\epsilon}\big)$, and the condition for transience is $\eta\geq\theta/(2pA^{**})$, $\theta>1$.

Now, let $s$ be such the constant that for almost all large $x$,
\[
\mathsf{E}\{|sX_{n+1}-sX_n|~\big|~sX_{n}=x, X_{n-1},\ldots, X_0\}=1.
\]
That is, due to \eqref{4.2}, $s=1/(2pA^{**})$.

In other words, the new stochastic sequence $X_n^\prime=sX_n$ belongs to the class $\frak{A}_{s^2\xi, s^2}$. Then, it is readily seen from the similar derivations provided before that
\[
\frac{p\mathsf{E}\{X_{n+1}^\prime-X_n^\prime~|~X_n^\prime=x, X_{n+1}^\prime> X_n^\prime, X_{n-1}^\prime,\ldots, X_0^\prime\}}{q\mathsf{E}\{X_{n}^\prime-X_{n+1}^\prime~|~X_n^\prime=x, X_{n+1}^\prime< X_n^\prime. X_{n-1}^\prime,\ldots, X_0^\prime\}}=1+\frac{\xi}{x}+O\left(\frac{1}{x^{1+\epsilon}}\right),
\]
for almost all large $x$

It follows from Corollary \ref{cor_Lamperti} that the stochastic sequence $X_n^\prime$ is recurrent, if $\xi\leq1+O\big(x^{-1-\epsilon}\big)$. If instead $\xi\geq\theta>1$, then the stochastic sequence $X_n^\prime$ is transient. Furthermore, recurrence (resp. transience) of $X_n^\prime$ implies recurrence (resp. transience) of stochastic sequence $X_n$, and vice versa.

We consider now a stochastic sequence $Y_n=RX_{n}^\prime$ belonging to the class $\frak{A}_{s^2R^2\xi, s^2R^2}$. Apparently that for almost all large $x$
\[
\begin{aligned}
&\mathsf{E}\{|Y_{n+1}-Y_n|~\big|~Y_{n}=x, Y_{n-1},\ldots, Y_0\}\\
&\quad=\mathsf{E}\{|RX_{n+1}^\prime-RX_n^\prime|~\big|~RX_{n}^\prime=x, X_{n-1}^\prime,\ldots, X_0^\prime\}\\
&\quad=R\mathsf{E}\{|X_{n+1}^\prime-X_n^\prime|~\big|~X_{n}^\prime=R^{-1}x, X_{n-1}^\prime,\ldots, X_0^\prime\}\\
&\quad=R+O\left(\frac{1}{x}\right).
\end{aligned}
\]

Together with this, for almost all large $x$ we have:
\[
\begin{aligned}
&\frac{p\mathsf{E}\{Y_{n+1}-Y_n~|~Y_n=x, Y_{n+1}>Y_n, Y_{n-1},\ldots, Y_0\}}{q\mathsf{E}\{Y_{n}-Y_{n+1}~|~Y_n=x, Y_{n+1}<Y_n, Y_{n-1},\ldots, Y_0\}}\\
&=\frac{p\mathsf{E}\{RX_{n+1}^\prime-RX_n^\prime~|~RX_n^\prime=x, X_{n+1}^\prime>X_n^\prime, X_{n-1}^\prime,\ldots, X_0^\prime\}}{q\mathsf{E}\{RX_{n}^\prime-RX_{n+1}^\prime~|~RX_n^\prime=x, X_{n+1}^\prime<X_n^\prime, X_{n-1}^\prime,\ldots, X_0^\prime\}}\\
&=\frac{p\mathsf{E}\{X_{n+1}^\prime-X_n^\prime~|~X_n^\prime=R^{-1}x, X_{n+1}^\prime>X_n^\prime, X_{n-1}^\prime,\ldots, X_0^\prime\}}{q\mathsf{E}\{X_{n}^\prime-X_{n+1}^\prime~|~X_n^\prime=R^{-1}x, X_{n+1}^\prime<X_n^\prime, X_{n-1}^\prime,\ldots, X_0^\prime\}}\\
&=1+\frac{R\xi}{x}+O\left(\frac{1}{x^{1+\epsilon}}\right).
\end{aligned}
\]
These last derivations finish the proof.
\end{proof}

\section{Discussion}\label{S3}
In this section, we discuss the connection of our result with the aforementioned result in \cite{A} given for the convenience of the readers in Appendix A.

In the proof of Theorem \ref{thm1}, we introduced the class $\frak{A}_{\xi, r^2}$ of processes. %
As it was mentioned in the proof of Theorem \ref{thm1}, there is one-to-one correspondence between elements of the class $\frak{A}_{r^2\xi, r^2}$ and elements of the class $\frak{A}_{\xi, 1}$. Furthermore, if a Markov chain belonging to $\frak{A}_{r^2\xi, r^2}$ is recurrent (resp. transient), then the same is true for the corresponding Markov chain belonging to $\frak{A}_{\xi, 1}$, and vice versa. Hence, one can restrict our attention by considering Markov chains belonging to the selected class $\frak{A}_{s^2\xi, s^2}$ that satisfies the property
\begin{equation}\label{d0}
\mathsf{E}\{|X_{n+1}-X_n|~\big|~X_{n}=x\}=1+O\left(\frac{1}{x}\right)
\end{equation}
for large $x$.

According to Theorem \ref{thm1}, the class of Markov chains $X_n$ belonging to $\frak{A}_{s^2\xi, s^2}$ is recurrent, if as $x\to\infty$,
\begin{equation}\label{d1}
\frac{p\mathsf{E}\{X_{n+1}-X_n~|~X_n=x, X_{n+1}>X_n\}}{q\mathsf{E}\{X_{n}-X_{n+1}~|~X_n=x, X_{n+1}<X_n\}}\leq1+\frac{1}{x}+O\left(\frac{1}{x^{1+\epsilon}}\right),
\end{equation}
where $p$ and $q$ are defined above under the assumptions of Theorem \ref{thm1}.
If instead
\begin{equation}\label{d2}
\frac{p\mathsf{E}\{X_{n+1}-X_n~|~X_n=x, X_{n+1}>X_n\}}{q\mathsf{E}\{X_{n}-X_{n+1}~|~X_{n}=x, X_{n+1}<X_n\}}\geq1+\frac{\theta}{x}
\end{equation}
for some $\theta>1$ and all large $x$, then the Markov chain is transient.

Under special assumptions made in \cite{A} (see Appendix A), it was proved that the Markov chain is recurrent if and only if
\[
\sum_{n=1}^\infty\prod_{i=1}^{n}\frac{e_i^-}{e_i^+}=\infty,
\]
where following the notation of this paper,
\begin{eqnarray*}
e_n^-&=&q\mathsf{E}\{X_{n}-X_{n+1}~|~X_{n+1}<X_n\},\\
e_n^+&=&p\mathsf{E}\{X_{n+1}-X_{n}~|~ X_{n+1}>X_n\}.
\end{eqnarray*}
That is, under the aforementioned additional assumptions (Appendix A), the necessary and sufficient condition for recurrence of the Markov chain is the divergence of the series
\begin{equation}\label{6}
\sum_{i=1}^\infty\prod_{n=1}^i\frac{q\mathsf{E}\{X_{n}-X_{n+1}~|~X_{n+1}<X_n\}}{p\mathsf{E}\{X_{n+1}-X_{n}~|~ X_{n+1}>X_n\}}.
\end{equation}

Under condition \eqref{d0}, the sufficient condition for divergence of \eqref{6} is \eqref{d1} and for convergence is \eqref{d2}. The condition for convergence \eqref{d2} follows from Raabe's ratio test \cite[p. 39]{BM} or particularly from Frink's ratio test \cite{F}, a version of Raabe's test that is more appropriate to the given case. For the condition for convergence \eqref{d1} we are to use Bertrand's test \cite[p. 40]{BM} rather than Raabe's test or Frink's test. This is because the presence of the remainder makes using the last two tests insufficient. For the readers convenience, the aforementioned tests of Raabe, Frink and Bertrand are presented in Appendix B.

The similar more extended sufficient conditions for recurrence or transience of birth-and-death processes can be found in \cite{A1}.

\subsection*{Statements and declarations}
\textit{Disclosure of interest.} No conflict of interests was reported by the author.

\textit{Declaration of funding.} No funding for this research was received.

\textit{Data availability statement.} Data sharing is not applicable to this article as no new data were created or analyzed in this study.

\appendix

\section{Recurrence of Markov chains with infinitely countable set of states}

Let $\mathscr{P}$ be irreducible Markov chain with the infinite set of states $\{0, 1,\ldots\}$, and let $p_{i,j}$ denote transition probabilities from state $i$ to state $j$.
\begin{defn}
The states of Markov chain $\mathscr{P}$ are said to form connected domain if for any $i\geq1$ there are $j_1(i)$ and $j_2(i)$, $0\leq j_1(i)<i<j_2(i)\leq\infty$ such that $p_{i,k}>0$ for all $j_1(i)\leq k\leq i-1$ and $i+1\leq k\leq j_2(i)$, and $\sum_{k=j_1(i)}^{j_2(i)}p_{i,k}=1.$
\end{defn}

\begin{thm}
Assume that the states of the Markov chain $\mathscr{P}$ form connected domain, and $\sum_{j=1}^\infty jp_{0,j}<\infty$. Denote $e^-=\sum_{j=1}^{\infty}jp_{i+j-1,i-1}$ and $e_i^+=\sum_{j=1}^{i+1}jp_{i-j+1,i+1}$, $i\geq1$, assuming that $e_i^-<\infty$. Assume also that $1-p_{i,i}>\epsilon$ for all $i\geq0$, where $\epsilon$ is some positive value. Then the Markov chain $\mathscr{P}$ is recurrent if and only if
\[
\sum_{n=1}^\infty\prod_{i=1}^{n}\frac{e_i^-}{e_i^+}=\infty.
\]
\end{thm}

\section{The tests of Raabe, Frink and Bertrand}

Let \[\sum_{n=1}^{\infty}a_n\]
be a series with positive terms.

\subsection*{Raabe's test} Let
\[
\rho_n\equiv\left(\frac{a_n}{a_{n+1}}-1\right).
\]

The series converges, if $\liminf_{n\to\infty}\rho_n>1$.

The series diverges, if $\limsup_{n\to\infty}\rho_n<1$.

Otherwise, the test is inconclusive.

\subsection*{Frink's test} The test is an alternative version of Raabe's test (see \cite{S}).

The series converges, if $\limsup_{n\to\infty}\big(\frac{a_{n+1}}{a_n}\big)^n<\frac{1}{\mathrm{e}}$.

The series diverges, if there is $N$ such that $\big(\frac{a_{n+1}}{a_n}\big)^n\geq\frac{1}{\mathrm{e}}$ for all $n\geq N$.

Otherwise, the test is inconclusive.

\subsection*{Bertrand's test} Let
\[
\rho_n\equiv n\ln n\left(\frac{a_n}{a_{n+1}}-1\right)-\ln n.
\]

The series converges, if $\liminf_{n\to\infty}\rho_n>1$.

The series diverges, if $\limsup_{n\to\infty}\rho_n<1$.

Otherwise, the test is inconclusive.

\end{document}